\documentclass[12pt,draft]{article} 
\usepackage{amsmath,amssymb,amsthm,amsfonts,bm}
\usepackage{enumerate,color}
\makeatletter
\def\@cite#1#2{[{{\bfseries #1}\if@tempswa , #2\fi}]}
\renewcommand{\section}{%
\@startsection{section}{1}{\z@}
{0.5truecm plus -1ex minus -.2ex}%
{1.0ex plus .2ex}{\bfseries\large}}
\def\@seccntformat#1{\csname the#1\endcsname.\ }
\makeatother

\numberwithin{equation}{section} 
\newtheorem{thm}{Theorem}[section]
\newtheorem{corollary}[thm]{Corollary}
\newtheorem{lem}[thm]{Lemma}

\theoremstyle{definition}

\newtheorem{remark}{Remark}[section]

\newcommand{\ep}{\varepsilon}
\newcommand{\pa}{\partial}
\newcommand{\Rn}{\mathbb{R}^n}
\newcommand{\uast}{u^{\ast}}
\newcommand{\vast}{v^{\ast}}
\newcommand{\wast}{w^{\ast}}
\newcommand{\ubar}{\overline{u}}
\newcommand{\vbar}{\overline{v}}
\newcommand{\wbar}{\overline{w}}
\newcommand{\Rone}{\mathbb{R}}

\newcommand{\ol}[1]{\overline{#1}}

\newcommand{\tmax}{T_{\rm max}}
\newcommand{\lp}[2]{\Vert{#2}\Vert_{L^{#1}(\Omega)}}
\newcommand{\wmp}[2]{\Vert{#2}\Vert_{W^{#1}(\Omega)}}
\newcommand{\cd}{(\cdot,t)}
\newcommand{\into}{\int_\Omega}

\begin{document}
\footnote[0]
    {2010{\it Mathematics Subject Classification}\/. 
    Primary: 35K51; Secondary: 92C17, 35B40.
    }
\begin{center}
    \Large{{\bf Boundedness and stabilization 
    in a two-species  
    chemotaxis-competition system
    of parabolic-parabolic-elliptic type
    }}
\end{center}
\vspace{5pt}
\begin{center}
    Masaaki Mizukami 
   \footnote[0]{
    E-mail: 
    {\tt masaaki.mizukami.math@gmail.com} 
    }\\
    \vspace{12pt}
    Department of Mathematics, 
    Tokyo University of Science\\
    1-3, Kagurazaka, Shinjuku-ku, 
    Tokyo 162-8601, Japan\\
    \vspace{2pt}
\end{center}
\begin{center}    
    \small \today
\end{center}

\vspace{2pt}
\newenvironment{summary}
{\vspace{.5\baselineskip}\begin{list}{}{%
     \setlength{\baselineskip}{0.85\baselineskip}
     \setlength{\topsep}{0pt}
     \setlength{\leftmargin}{12mm}
     \setlength{\rightmargin}{12mm}
     \setlength{\listparindent}{0mm}
     \setlength{\itemindent}{\listparindent}
     \setlength{\parsep}{0pt}
     \item\relax}}{\end{list}\vspace{.5\baselineskip}}
\begin{summary}
{\footnotesize {\bf Abstract.}
This paper deals with the two-species 
chemotaxis-competition system 
\begin{equation*}
  \begin{cases}
    u_t=d_1\Delta u - \chi_1\nabla \cdot (u \nabla w)
    +\mu_1 u(1-u-a_1 v) 
    & {\rm in} \ \Omega \times (0,\infty), \\
    v_t=d_2\Delta v - \chi_2\nabla \cdot (v \nabla w)
    +\mu_2 v(1-a_2u-v) 
    & {\rm in} \ \Omega \times (0,\infty), \\
    0=d_3\Delta w + \alpha u + \beta v - \gamma w  
    & {\rm in} \ \Omega \times (0,\infty),
   \end{cases}
\end{equation*}
where $\Omega$ is a bounded domain in $\Rn$ 
with smooth boundary $\pa \Omega$, $n\ge 2$; 
$\chi_i$ and $\mu_i$ are constants satisfying 
some conditions. The above system was 
studied in the cases that 
$a_1,a_2\in (0,1)$ and $a_1>1>a_2$, 
and it was proved that 
global existence and asymptotic stability hold 
when $\frac{\chi_i}{\mu_i}$ are small 
(\cite{Black-Lankeit-Mizukami_01,stinner_tello_winkler,Tello_Winkler_2012}). 
However, the conditions in the above two cases 
strongly depend on $a_1,a_2$, and 
have not been obtained in the case that $a_1,a_2\ge 1$. 
Moreover, convergence rates in the cases 
that $a_1,a_2\in (0,1)$ and $a_1 > 1 > a_2$ 
have not been studied. 
The purpose of this work is to construct conditions 
which derive 
global existence of classical bounded solutions 
for all $a_1,a_2>0$ 
which covers the case that $a_1,a_2 \ge 1$, 
and lead to convergence rates 
for solutions of the above system 
in the cases that 
$a_1,a_2\in (0,1)$ and $a_1\ge 1 >a_2$. 
}
\end{summary}
\vspace{10pt}

\newpage
%
%

\section{Introduction}

Many phenomena, 
which appear in natural science, especially, 
biology, chemistry and physics, 
support animals' 
lives. 
In this paper we focus on 
{\it chemotaxis} 
which is one of the important 
properties and is related to e.g., movement of sperm, 
migration of neurons or lymphocytes and tumor invasion. 
Chemotaxis is the property such that species move 
towards higher concentration of the chemical substance 
when they plunge into hunger. 

A mathematical problem which describes a part of the 
life cycle of cellular slime molds with chemotaxis 
is called the Keller--Segel system: 
\[
u_t = \Delta u - \chi \nabla \cdot (u\nabla v), 
\quad 
\tau v_t = \Delta v + u -v, 
\]
where $\chi>0$ and $\tau\in\{0,1\}$. 
Moreover, the chemotaxis system 
with growth terms 
\[
u_t = \Delta u + \chi \nabla \cdot (u\nabla v) 
      + \kappa u-\mu u^2,
\quad 
\tau v_t = \Delta v + u -v 
\] 
was proposed by \cite{mimura_tsujikawa,OTYM}, 
where $\chi,\kappa,\mu>0$ and $\tau\in \{0,1\}$. 
After the pioneering work of Keller--Segel \cite{K-S}, 
the Keller--Segel system and 
the chemotaxis system are intensively studied 
(see e.g., \cite{B-B-T-W,Hillen_Painter_2009,Horstmann_2003}). 
A generalized problem 
of Keller--Segel systems, 
which means a two-species chemotaxis system, 
was proposed in \cite{wolansky} 
and also has studied 
(see e.g., 
\cite{biler_espejo_guerra,
biler_guerra,
Conca-Espejo,
Conca-Espejo-Viliches,
espejo_stevens_velazquez_simultaneous,
lili, 
Zhang-Li_2014}; 
global existence 
was proved in \cite {Conca-Espejo,Conca-Espejo-Viliches,Zhang-Li_2014}; 
and thier asymptotic stability was shown in 
\cite{Zhang-Li_2014}; related works which deal with 
blow-up of solutions can be seen in 
\cite{biler_espejo_guerra,
biler_guerra,
Conca-Espejo,
Conca-Espejo-Viliches,
espejo_stevens_velazquez_simultaneous,
lili}). 
Recently, a two-species chemotaxis system with 
competitive kinetics 
\begin{gather*}
    u_t=\Delta u - \chi_1\nabla \cdot (u \nabla w)
    +\mu_1 u(1-u-a_1 v), 
\\
    v_t=\Delta v - \chi_2\nabla \cdot (v \nabla w)
    +\mu_2 v(1-a_2u-v),
\\
    \tau w_t = \Delta w + \alpha u + \beta v - \gamma w 
\end{gather*}
with some $\chi_1,\chi_2,\mu_1,\mu_2,a_1,a_2 >0$ 
and $\tau \in \{0,1\}$, 
which describes the evolution of 
two competing species which react 
on a single chemoattractant, 
was proposed 
by Tello--Winkler \cite{Tello_Winkler_2012} 
and was studied 
(see \cite{Bai-Winkler_2016,Lin-Mu-Wang,Mizukami_AIMSmath,Mizukami_DCDSB,Mizukami-Yokota_01,N-T_SIAM,N-T_JDE,Zhang-Li_2015}). 
About this problem 
with $\tau=1$, 
global existence and boundedness 
was obtained in the 2-dimensional 
case (\cite{Bai-Winkler_2016}) 
and the $n$-dimensional setting (\cite{Lin-Mu-Wang}); 
moreover, asymptotic behavior of solutions 
was established in \cite{Bai-Winkler_2016, Mizukami_DCDSB}. 
Related works which dealt with global existence and 
boundedness in this two-species problem 
with sensitivity functions 
can be found in \cite{Mizukami_DCDSB,Zhang-Li_2015}; 
and related works which treated 
the non-competition case are in \cite{Mizukami_AIMSmath,Mizukami-Yokota_01,N-T_SIAM,N-T_JDE}. 
These results in the case $\tau =1$ 
are motivated by the results (\cite{Black-Lankeit-Mizukami_01,
stinner_tello_winkler,
Tello_Winkler_2012}) in the case $\tau=0$. 
Therefore the parabolic-parabolic-elliptic problem 
reduced by letting $\tau=0$ 
seems to be helpful to analyze the fully parabolic case.  

%
%
In this paper we consider 
the two-species chemotaxis 
system with competitive kinetics 
of parabolic-parabolic-elliptic type
\begin{equation}\label{cp}
  \begin{cases}
    u_t=d_1\Delta u - \chi_1\nabla \cdot (u \nabla w)
    +\mu_1 u(1-u-a_1 v), 
    & x\in\Omega,\ t>0, 
\\[1mm]
    v_t=d_2\Delta v - \chi_2\nabla \cdot (v \nabla w)
    +\mu_2 v(1-a_2u-v), 
    & x\in\Omega,\ t>0,  
\\[1mm]	
    0=d_3\Delta w + \alpha u + \beta v - \gamma w, 
    & x\in\Omega,\ t>0, 
\\[1mm]
    \nabla u\cdot \nu=\nabla v\cdot \nu = \nabla w\cdot \nu = 0, 
    & x\in\pa \Omega,\ t>0,
\\[1mm]
    u(x,0)=u_0(x),\; v(x,0)=v_0(x),
    & x\in\Omega,
  \end{cases}
\end{equation}
where $\Omega$ is a bounded domain in $\Rn$ 
($n\ge 2$) 
with smooth boundary $\pa \Omega$ 
and $\nu$ 
is 
the 
outward normal vector to $\pa\Omega$. 
The constants $d_1,d_2,d_3,\chi_1,\chi_2,\mu_1,\mu_2, 
a_1,a_2$ 
and $\alpha,\beta,\gamma$ 
are positive. 
The initial data $u_0,v_0$ are
assumed to be nonnegative functions.
The unknown functions $u(x,t)$ and $v(x,t)$  
represent the population densities of 
two species and 
$w(x,t)$ shows the concentration of the 
chemical substance 
at place $x$ and time $t$. 

The problem \eqref{cp} is a problem 
on account of the influence of chemotaxis, diffusion, 
and the Lotka--Volterra competitive kinetics, 
i.e., with coupling coefficients $a_1,a_2>0$ in 
\begin{align}\label{L-V}
  u_t = u(1-u-a_1 v), \quad v_t = v(1-a_2u-v). 
\end{align}
The mathematical difficulties of the problem \eqref{cp} 
are to deal with the chemotaxis term $\nabla \cdot (u\nabla w)$ 
and the competition term $u (1-u-a_1v)$. 
To overcome these difficulties, 
in the case that $a_1,a_2\in (0,1)$ and 
$d_3=\alpha=\beta=1$ in \eqref{cp}, 
Tello--Winkler \cite{Tello_Winkler_2012} applied 
comparison methods to this problem and obtained 
global existence of classical bounded solutions 
and their asymptotic behavior 
under the conditions that 
\begin{align}\label{condition;pre1}
2(\chi_1+\chi_2) + a_2\mu_1  < \mu_2
\quad \mbox{and}\quad 
2(\chi_1 +\chi_2) + a_1\mu_2  < \mu_1. 
\end{align}
However, if $\chi_1\to 0$ or $\mu_1\to \infty$, 
then these conditions break down. 
Recently, it was shown that the conditions 
\begin{align}\label{condition;pre2}
  &\frac{\chi_1}{\mu_1} 
  \in \left[0,\frac{d_3}{2\alpha}\right)
      \cap \left[0,\frac{a_1d_3}{\beta}\right), 
  \quad
  \frac{\chi_2}{\mu_2} 
  \in \left[0,\frac{d_3}{2\beta}\right)
      \cap \left[0,\frac{a_2d_3}{\alpha}\right), 
\\\label{condition;pre3}
  &a_1a_2d_3^2 < 
  \left(d_3-\frac{2\alpha\chi_1}{\mu_1}\right)
  \left(d_3-\frac{2\beta\chi_2}{\mu_2}\right)
\end{align}
lead to global existence and asymptotic stability 
in \eqref{cp} in the case that $a_1,a_2\in (0,1)$ 
(\cite{Black-Lankeit-Mizukami_01}). 
The conditions \eqref{condition;pre2}--\eqref{condition;pre3} 
partially relax \eqref{condition;pre1} 
in view of the point mentioned above. 
On the other hand, 
in the case that $a_1>1>a_2$ and $d_3=\beta=1$ in \eqref{cp} 
Stinner--Tello--Winkler \cite{stinner_tello_winkler} 
established global existence and stabilization 
of global classical solutions when 
\begin{align*} 
  &\frac{\chi_1}{\mu_1}\le a_1, \quad 
  \frac{\chi_2}{\mu_2}< \frac 12 
  \quad \mbox{and}
\\ 
  &\frac{\alpha \chi_1}{\mu_1} 
  + \max\left\{
  \frac{\chi_2}{\mu_2}, 
  \frac{a_2(\mu_2-\chi_2)}{\mu_2-2\chi_2}, 
  \frac{(\alpha - a_2)\chi_2}{\mu_2-2\chi_2}
  \right\} 
  <1 
\end{align*} 
are satisfied. 
In summary the two-species chemotaxis-competition model 
\eqref{cp} were studied in the cases that 
$a_1,a_2\in (0,1)$ and $a_1>1>a_2$, 
and it was proved that 
global existence and same asymptotic behavior as 
solutions to the 
Lotka--Volterra competition model \eqref{L-V} 
hold when $\frac{\chi_i}{\mu_i}$ are small. 
However, the conditions in the above two cases 
strongly depend on $a_1,a_2$, and 
have not been obtained in the case that $a_1,a_2\ge 1$. 
Moreover, convergence rates in the cases 
that $a_1,a_2\in (0,1)$ and $a_1 > 1> a_2$ 
have not been studied. 

The purpose of this work is to construct conditions 
which derive 
global existence of classical bounded solutions 
for all $a_1,a_2>0$ 
which covers the case that $a_1,a_2 \ge 1$, 
and lead to convergence rates 
for solutions of \eqref{cp} 
in the cases that 
$a_1,a_2\in (0,1)$ and $a_1 > 1 > a_2$. 


%
%
%
For establishing global existence and boundedness 
we shall suppose 
that $\chi_1,\chi_2$ and $\mu_1,\mu_2$ satisfy 
the following conditions: 
\begin{align}\label{condition;bounded}
  \frac{\chi_1}{\mu_1} < \frac{nd_3}{n-2}
  \min\left\{\frac{1}{\alpha},\frac{a_1}{\beta}\right\} 
  \quad 
  \mbox{and} 
  \quad 
  \frac{\chi_2}{\mu_2} < \frac{nd_3}{n-2}
  \min\left\{\frac{1}{\beta},\frac{a_2}{\alpha}\right\}.
\end{align}
We assume that the initial data $u_0, v_0$ 
satisfy
\begin{align}\label{ini} 
0\leq u_0\in C(\ol{\Omega})\setminus \{0\}, \quad
0\leq v_0\in C(\ol{\Omega})\setminus \{0\}. 
\end{align} 

%
%
Now the main results read as follows. 
The first one is concerned 
with global existence and 
boundedness in \eqref{cp}.  
%
\begin{thm}\label{mainth} 
 Let $d_1,d_2,d_3> 0$, $\mu_1,\mu_2 > 0$, 
 $a_1,a_2> 0$, $\chi_1,\chi_2>0$, 
 $\alpha,\beta,\gamma>0$ and 
 let $\Omega \subset \Rn$ $(n\ge 2)$ be a bounded domain 
 with smooth boundary. 
 Assume that \eqref{condition;bounded} are satisfied. 
 Then for any $u_0, v_0$ 
 satisfying \eqref{ini} with some $q>n$, 
 there exists an exactly one pair $(u,v,w)$ of 
 nonnegative functions 
 \begin{align*} 
   &u,\; v\in C(\ol{\Omega}\times [0,\infty))\cap 
   C^{2,1}(\ol{\Omega}\times (0,\infty)), 
   \\
   &w\in C(\ol{\Omega}\times [0,\infty))\cap 
   C^{2,1}(\ol{\Omega}\times (0,\infty))\cap 
   L^\infty_{\rm loc}([0,\infty);W^{1,q}(\Omega)), 
 \end{align*} 
 which satisfy \eqref{cp}. 
%
Moreover, the solutions $u,v,w$ are 
uniformly bounded, 
i.e., there exists a constant $C>0$ such that
\begin{align*}
  \lp{\infty}{u\cd}+\lp{\infty}{v\cd}+
  \wmp{1,q}{w\cd}\leq C\quad 
  \mbox{for all}\ t\geq 0, 
\end{align*}
and the solutions $u,v,w$ are 
the H\"older continuous functions, i.e., 
there exist $\theta\in(0,1)$ and $M>0$ such that 
\begin{align*}
   \|u\|_{C^{\theta,\frac{\theta}{2}}
   (\ol{\Omega}\times[t,t+1])}
   +\|v\|_{C^{\theta,\frac{\theta}{2}}
   (\ol{\Omega}\times[t,t+1])}
   +\|w\|_{C^{\theta,\frac{\theta}{2}}
   (\ol{\Omega}\times[t,t+1])}\leq M
   \quad \mbox{for all}\ t\geq 1.
\end{align*}
\end{thm}
%
\begin{remark}
This result give the existence of global classical 
bounded solutions in the case that $a_1,a_2\ge 1$. 
Moreover, 
the condition \eqref{condition;bounded} relaxes 
\eqref{condition;pre2} which assumed for 
global existence of solutions in \cite{Black-Lankeit-Mizukami_01}. 
Indeed, if $\chi_1,\chi_2$ and $\mu_1,\mu_2$ 
satisfy the condition \eqref{condition;pre2}, 
then $\chi_1,\chi_2$ and $\mu_1,\mu_2$ 
satisfy the condition \eqref{condition;bounded}. 
However, the condition \eqref{condition;bounded} 
does not always relax 
those assumed in \cite{stinner_tello_winkler} 
and \cite{Tello_Winkler_2012}; in the case that 
$a_1,a_2\in (0,1)$ the condition 
\eqref{condition;bounded} 
relaxes \eqref{condition;pre1} 
under the condition 
\begin{align*}
  \chi_1 < \frac{2na_1(\chi_1+\chi_2)(1+a_1)}{(n-2)(1-a_1a_2)} 
  \quad \mbox{and}\quad 
  \chi_2 < \frac{2na_2(\chi_1+\chi_2)(1+a_2)}{(n-2)(1-a_1a_2)}, 
\end{align*} 
and in the case that $a_1> 1 >a_2$ 
the condition \eqref{condition;bounded} relaxes 
the condition 
\begin{align*}
  \frac{\alpha \chi_1}{\mu_1} + \frac{\chi_2}{\mu_2} < 1, 
\end{align*}
which was used to obtain global existence 
in \cite{stinner_tello_winkler}, 
when 
\[
  \frac{\alpha (n-2)}{n}
  < 
  \min\{
  1,
  a_1\alpha, 
  a_2 
  \}
\]
hold. 
\end{remark}

The main theorem tells us the following result 
in the 2-dimensional case. 

\begin{corollary}
 Let $d_1,d_2,d_3> 0$, $\mu_1,\mu_2 > 0$, 
 $a_1,a_2> 0$, $\chi_1,\chi_2>0$, 
 $\alpha,\beta,\gamma>0$ and 
 let $\Omega \subset \mathbb{R}^2$ 
 be a bounded domain 
 with smooth boundary. 
 Then for any $u_0,v_0$ satisfying \eqref{ini} 
 with some $q>n$, \eqref{cp} possesses a unique 
 global bounded classical solution. 
\end{corollary}

In the case $a_1,a_2\in(0,1)$ 
asymptotic behavior of solutions to \eqref{cp} 
will be discussed under the following additional 
conditions: 
there exists $\delta_1>0$ such that 
\begin{align}\label{case1condition1}
  &4\delta_1-a_1a_2(1+\delta_1)^2>0
  \end{align}
  and 
  \begin{align}\label{case1condition2}
  &\mu_1
  >
  \frac{\chi_1^2(1+\delta_1)(1-a_1)(\alpha^2a_1\delta_1
    +\beta^2a_2-\alpha\beta a_1a_2(1+\delta_1))}
  {4a_1d_1d_3\gamma(1-a_1a_2)(4\delta_1-a_1a_2(1+\delta_1)^2)},
\\\label{case1condition3}
    &\mu_2 > 
    \frac{\chi_2^2(1+\delta_1) 
    (1-a_2)(\alpha^2a_1\delta_1 
    +\beta^2a_2-\alpha \beta a_1a_2(1+\delta_1))} 
    {4a_2d_2d_3\gamma(1-a_1a_2)(4\delta_1-a_1a_2(1+\delta_1)^2)}. 
  \end{align}
The second theorem gives asymptotic behavior 
in \eqref{cp} 
in the case $a_1,a_2\in (0,1)$. 
\begin{thm}\label{mainth2} 
Let $d_1,d_2,d_3> 0$, $\mu_1,\mu_2 > 0$, 
$a_1,a_2\in (0,1)$, $\chi_1,\chi_2>0$, 
$\alpha,\beta,\gamma>0$ and 
let $\Omega \subset \Rn$ $(n\ge 2)$ be a bounded domain 
with smooth boundary. 
Assume that there exists a unique global 
classical solution $(u,v,w)$ of \eqref{cp} 
satisfying 
\begin{align*}
 \|u\|_{C^{\theta,\frac{\theta}{2}}
 (\ol{\Omega}\times[t,t+1])} 
 + \|v\|_{C^{\theta,\frac{\theta}{2}}
 (\ol{\Omega}\times[t,t+1])} 
 + \|w\|_{C^{\theta,\frac{\theta}{2}}
 (\ol{\Omega}\times[t,t+1])}
 \le M
 \quad \mbox{for all}\ t>0
\end{align*}
with some $M>0$. 
Then 
under the conditions 
\eqref{case1condition1}{\rm --}\eqref{case1condition3}, 
$(u,v,w)$ satisfies that there exist $C>0$ 
and $\ell>0$ such that 
\begin{align*}
 \|u\cd-u^\ast\|_{L^{\infty}(\Omega)} 
 +\|v\cd-v^\ast\|_{L^{\infty}(\Omega)} 
 +\|w\cd-w^\ast\|_{L^{\infty}(\Omega)} 
 \leq Ce^{-\ell t} 
 \quad 
  \mbox{for all}\ t>0, 
\end{align*} 
where 
\[
  u^\ast:=\frac{1-a_1}{1-a_1a_2},
\quad 
  v^\ast:=\frac{1-a_2}{1-a_1a_2}
\quad 
  w^\ast := \frac{\alpha \uast +\beta \vast}{\gamma}.
\]
\end{thm}
%
\begin{remark}
If the assumption of Theorem \ref{mainth} and 
\eqref{case1condition1}{\rm --}\eqref{case1condition3} 
are satisfied, then 
Theorem \ref{mainth2} gives 
the convergence rates for solutions 
of \eqref{cp} in the case that $a_1,a_2\in (0,1)$. 
Moreover, the conditions 
\eqref{case1condition1}{\rm --}\eqref{case1condition3} 
are the same conditions as that assumed in \cite{Mizukami_DCDSB} 
in the case that $a_1,a_2\in (0,1)$ and 
$h(u,v,w)=\alpha u + \beta v -\gamma w$. 
\end{remark}

In the case $a_1\geq 1>a_2$ 
asymptotic behavior of solutions to \eqref{cp} 
will be discussed under the following additional 
conditions: 
there exist $\delta_1>0$ and $a'_1\in[1,a_1]$ 
such that 
\begin{align}\label{case2condition1}
  &4\delta_1-a'_1a_2(1+\delta_1)^2>0, 
\\\label{case2condition2}
  &{\mu_2}>\frac{\chi_2^2\delta_1
  (\alpha^2a'_1\delta_1
  +\beta^2a_2-\alpha\beta a'_1a_2(1+\delta_1))}
  {4a_2d_2d_3\gamma(4\delta_1-a'_1a_2(1+\delta_1)^2)}.
\end{align}
The third one gives asymptotic behavior 
in \eqref{cp} 
in the case $a_1\geq 1> a_2$. 
%
\begin{thm}\label{mainth3} 
 Let $d_1,d_2,d_3> 0$, $\mu_1,\mu_2 > 0$, 
 $a_1\ge 1$, $a_2\in (0,1)$, $\chi_1,\chi_2>0$, 
 $\alpha,\beta,\gamma>0$ and 
 let $\Omega \subset \Rn$ $(n\ge 2)$ be a bounded domain 
 with smooth boundary. 
Assume that there exists a unique global 
classical solution $(u,v,w)$ of \eqref{cp} 
such that 
\begin{align*}
 \|u\|_{C^{\theta,\frac{\theta}{2}}
 (\ol{\Omega}\times[t,t+1])} 
 + \|v\|_{C^{\theta,\frac{\theta}{2}}
 (\ol{\Omega}\times[t,t+1])} 
 + \|w\|_{C^{\theta,\frac{\theta}{2}}
 (\ol{\Omega}\times[t,t+1])}
 \le M
 \quad \mbox{for all}\ t>0
\end{align*}
with some $M>0$. 
Then under the conditions 
\eqref{case2condition1}{\rm --}\eqref{case2condition2}, 
$(u,v,w)$ has the following properties\/{\rm :} 
\begin{enumerate}
\item[{\rm (i)}] 
If $a_1>1$ and take $a'_1\in (1,a_1]$ 
in \eqref{case2condition1}--\eqref{case2condition2}, 
then there exist $C>0$ and $\ell > 0$ satisfying 
  \begin{align*}
  \|u(t)\|_{L^{\infty}(\Omega)}
  + \|v(t)-1\|_{L^{\infty}(\Omega)}
  +\left\|w(t)-\frac{\beta}{\gamma}\right\|_{L^{\infty}(\Omega)}
  \leq Ce^{-\ell t}
  \quad \mbox{for all}\ t>0. 
  \end{align*}
\item[{\rm (ii)}]
If $a_1=1$, 
then there exist $C>0$ and $\ell>0$ satisfying 
  \begin{align*}
  \|u(t)\|_{L^{\infty}(\Omega)}
  + \|v(t)-1\|_{L^{\infty}(\Omega)}
  +\left\|w(t)-\frac{\beta}{\gamma}\right\|_{L^{\infty}(\Omega)}
  \leq C(t+1)^{-\ell}
  \quad \mbox{for all}\ t>0. 
  \end{align*}
\end{enumerate}
\end{thm}
%
\begin{remark}
If the assumption of Theorem \ref{mainth} and 
\eqref{case2condition1}{\rm --}\eqref{case2condition2} 
are satisfied, then 
Theorem \ref{mainth3} gives 
the convergence rates for 
solutions in the cases that $a_1 > 1 > a_2$ and $a_1 = 1 > a_2$. 
Moreover, the conditions 
\eqref{case2condition1}{\rm --}\eqref{case2condition2} 
are the same conditions as that assumed in \cite{Mizukami_DCDSB} 
in the case that $a_1\ge 1 >a_2$ and 
$h(u,v,w)=\alpha u+\beta v -\gamma w$. 
\end{remark}

\begin{remark} 
Stabilization in the case that $a_1,a_2\ge 1$ 
is a still open question. 
In the case that $a_1,a_2>1$ 
a Lotka--Volterra competition model 
with diffusion term was studied; 
however, its analysis is difficult 
and it is known that 
solutions have complicated structures 
(see cf. \cite{Ei-Yanagida_1994,
Iida-Muramatsu-Ninomiya-Yanagida,
Kan-on_1995,
Kan-on-Yanagida,
Kishimoto-Weinberger_1985,
Matano-Mimura,
Mimura-Ei-Fang_1991}). 
\end{remark}

The strategy of the proof of Theorem \ref{mainth} 
is to extend a method in \cite{Tello-Winkler_2007} 
to a two-species case. 
We first aim to establish the $L^p$-estimate 
for $u$ with some $p>\frac n2$ 
from the following derivative of $\into u^p$: 
\begin{align}\label{derofu^p}
 \frac{1}{p}\frac{d}{dt}\into u^p
 \le (p-1)\chi_1\into u^{p-1}\nabla u\cdot \nabla v
 + \mu_1 \into u^p(1-u-a_1v). 
\end{align}
Since the third equation in \eqref{cp} derives that 
\begin{align}\label{lapwchange} 
(p-1)\chi_1 \into u^{p-1} \nabla u\cdot \nabla v 
  &= \frac{(p-1)\chi_1}{d_3p} 
  \into u^p (\alpha u+\beta v - \gamma w),
\end{align}
we shall show that a combination of 
\eqref{derofu^p} and \eqref{lapwchange}, 
along with the condition \eqref{condition;bounded} 
implies 
\begin{align*}
  \frac{d}{dt}\into u^p 
  \le -c_1 \left(\into u^p\right)^{\frac{p+1}{p}} 
  +c_2\into u^p,
\end{align*}
which leads to $L^p$-estimate for $u$. 
Then aided by standard semigroup estimates, 
we can obtain the $L^\infty$-estimate 
for $u$. 
On the other hand, one of the keys for 
the proof of Theorems \ref{mainth2} and \ref{mainth3} 
is to derive the following energy estimate: 
\begin{align}\label{strategy;stab}
 \frac{d}{dt}E(t) \le 
 -\ep \into 
 \left[
   (u\cd-\ubar)^2 + (v\cd-\vbar)^2 + (w\cd-\wbar)^2
 \right]
\end{align}
for all $t>0$ with some positive function $E$ and 
some constant $\ep>0$, 
where $(\ubar,\vbar,\wbar)\in \mathbb{R}^3$ is 
a solution of \eqref{cp}. 
Thanks to \eqref{strategy;stab}, we can obtain that 
there exists $C>0$ such that 
\begin{align*}
 \int_0^\infty\into (u-\ubar)^2 
 + \int_0^\infty\into (v-\vbar)^2
 + \int_0^\infty\into (w-\wbar)^2
 \le C, 
\end{align*}
which together with the regularity of the solution 
leads to Theorems \ref{mainth2} and \ref{mainth3}. 

%
%

This paper is organized as follows. 
In Section 2 
we prove 
global existence 
and boundedness (Theorem \ref{mainth}) 
through a series of lemmas. 
Section 3 is devoted to the proof of 
asymptotic stability 
(Theorems \ref{mainth2} and \ref{mainth3}); 
we first provide some lemmas which will be used later, 
and we next devide the section into Sections 3.1 and 3.2 
according to the proof of Theorem \ref{mainth2} and 
that of Theorem \ref{mainth3}, respectivly. 

%
%

\section{Global existence and boundedness}
In this section 
we shall show global existence and 
boundedness in \eqref{cp}. 
First we will recall the known result 
about local existence of solutions to \eqref{cp} 
(\cite[Lemma 2.1]{Black-Lankeit-Mizukami_01}, 
\cite[Lemma 2.1]{stinner_tello_winkler}).

%
%

\begin{lem}\label{localexistence}
 Let $d_1,d_2,d_3> 0$, $\mu_1,\mu_2 > 0$, 
 $a_1,a_2> 0$, $\chi_1,\chi_2>0$, 
 $\alpha,\beta,\gamma>0$ and 
 let $\Omega \subset \Rn$ $(n\in \mathbb{N})$ be a bounded domain 
 with smooth boundary. 
 Then for any $u_0, v_0$ satisfying \eqref{ini} 
 for some $q>n$, 
 there exist $\tmax \in$ $(0,\infty]$ and an 
 exactly one pair $(u,v,w)$ of nonnegative functions 
 \begin{align*}
   &u,\; v,\; w\in C(\ol{\Omega}\times [0,\tmax))\cap 
   C^{2,1}(\ol{\Omega}\times (0,\tmax))
 \end{align*}
 which satisfy \eqref{cp}. 
 Moreover, 
 \begin{align*}
   either\ \tmax=\infty \quad 
   or \quad \lim_{t\to \tmax}
    (\lp{\infty}{u\cd}+
    \lp{\infty}{v\cd})=\infty.
\end{align*}
\end{lem}
%
%
%
%
%
%
%
%
%
%
%
We next give the $L^p$-estimate for $u$ with some $p>\frac n2$ 
which plays an important role in deriving 
$L^\infty$-estimate for $u$. 
The proof is based on the proof of 
\cite[Lemma 2.2]{Tello-Winkler_2007}. 
\begin{lem}\label{lem;Lp;u}
Assume that 
\eqref{condition;bounded}{\rm --}\eqref{ini} 
are satisfied. 
Then for all 
$p\in I_1$, 
there exists $C(p)>0$ such that 
\begin{align*}
 \lp{p}{u\cd} \le C(p)
\end{align*}
for all $t>0$, where 
\[
I_1 := \left(
\frac n2, 
\min
\left\{
\frac{\alpha\chi_1}{(\alpha\chi_1-d_3\mu_1)_+}, 
\frac{\beta \chi_1}{(\beta \chi_1 - a_1 d_3 \mu_1)_+}
\right\}
\right).
\]
\end{lem}
\begin{proof}
We fix $p\in I_1$. 
Here we note from 
the condition \eqref{condition;bounded} that 
$
I_1
\neq \emptyset$.
Multiplying the first equation in \eqref{cp} by 
$u^{p-1}$ and integrating it over $\Omega$, 
we obtain that 
\begin{align}\notag
  &\frac 1p\frac{d}{dt} \into u^p 
  + d_1 (p-1)\into u^{p-2}|\nabla u|^2 
\\\label{eq;up}
  &= (p-1)\chi_1 \into u^{p-1} \nabla u\cdot \nabla v 
  + \mu_1 \into u^p (1-u-a_1 v). 
\end{align}
Then integration by parts and 
the third equation in \eqref{cp} imply that 
\begin{align}\notag
  (p-1)\chi_1 \into u^{p-1} \nabla u\cdot \nabla v 
  &= -\frac{(p-1)\chi_1}{p}\into u^p \Delta v
  \\\label{equ;nablaunablav}
  &= \frac{(p-1)\chi_1}{d_3p} 
  \into u^p (\alpha u+\beta v - \gamma w). 
\end{align}
Therefore a combination of \eqref{eq;up} 
with \eqref{equ;nablaunablav} yields that 
\begin{align*}
  \frac 1p \frac{d}{dt}\into u^p 
  \le 
    \mu_1 \into u^p
    -\left(\mu_1-\frac{\alpha(p-1)\chi_1}{d_3 p}\right)
    \into u^{p+1} 
    -\left(a_1\mu_1-\frac{\beta(p-1)\chi_1}{d_3 p}\right)
    \into u^pv. 
\end{align*}
Recalling %
$p\in I_1=
\left(\frac n2, 
\min\left\{
\frac{\alpha\chi_1}{(\alpha\chi_1-d_3\mu_1)_+}, 
\frac{\beta \chi_1}{(\beta \chi_1 - a_1 d_3 \mu_1)_+}
\right\}\right)$ that 
\begin{align*}
  \mu_1-\frac{\alpha(p-1)\chi_1}{d_3 p}>0 
  \quad \mbox{and}\quad 
  a_1\mu_1-\frac{\beta(p-1)\chi_1}{d_3 p}>0, 
\end{align*}
we establish from the H\"older inequality 
\[
  \into u^p 
   \le |\Omega|^{\frac{p}{p+1}}
   \left(\into u^{p+1}\right)^{\frac{p}{p+1}}
\]
that there exists $\ep>0$ satisfying 
\begin{align*}
    \frac 1p \frac{d}{dt}\into u^p 
  \le -\ep\left(\into u^p\right)^{\frac{p+1}{p}} + \mu_1 \into u^p, 
\end{align*} 
which implies that 
\begin{align*} 
  \lp{p}{u\cd}\le 
  \min\left\{\lp{p}{u_0},\frac{\mu_1}{\ep}\right\} 
  \quad \mbox{for all}\ t\in (0,\tmax). 
\end{align*} 
Thus we can attain the conclusion of this lemma. 
\end{proof}

%
%
%
%
Similarly, we can confirm the 
following $L^p$-estimate for $v$ 
with some $p>\frac n2$. 
\begin{lem}\label{lem;Lp;v}
Assume that 
\eqref{condition;bounded}{\rm --}\eqref{ini} 
are satisfied. 
Then for all 
$p\in I_2$, 
there exists $C(p)>0$ such that 
\begin{align*}
 \lp{p}{v\cd} \le C(p)
 \quad\mbox{for all}\ t>0, 
\end{align*}
where 
$I_2 := \left(
\frac n2, 
\min
\left\{
\frac{\beta\chi_2}{(\beta\chi_2-d_3\mu_2)_+}, 
\frac{\alpha \chi_2}{(\alpha\chi_2 - a_2 d_3 \mu_2)_+}
\right\}
\right)$. 
\end{lem}
\begin{proof}
A similar argument as in the proof of 
Lemma \ref{lem;Lp;v} derives this lemma. 
\end{proof}

%
%
%
%
Now we could construct all estimates which 
will enable us to obtain the estimate 
for the solution; Lemmas \ref{lem;Lp;u} 
and \ref{lem;Lp;v} lead to the following lemma. 
The proof is based on a known argument 
involving semigroup estimates which 
derive the $L^\infty$-estimate for $u$ 
from $L^p$-estimate 
with $p>\frac n2$ (see e.g., \cite{B-B-T-W}). 
\begin{lem}\label{lem;Linfty;uvw}
Assume that 
\eqref{condition;bounded}{\rm --}\eqref{ini} 
are satisfied. 
Then there exists $C > 0$ such that 
\begin{align}\label{Linfty;uvandw}
 \lp{\infty}{u\cd} 
 +\lp{\infty}{v\cd}
 +\|w\cd\|_{W^{1,q}(\Omega)}
 \le C 
 \quad \mbox{for all}\ t>0. 
\end{align}
Moreover, there exist $M > 0$ and $\theta\in (0,1)$ 
such that 
\begin{align*}
   \|u\|_{C^{\theta,\frac{\theta}{2}}
   (\ol{\Omega}\times[t,t+1])}
   +\|v\|_{C^{\theta,\frac{\theta}{2}}
   (\ol{\Omega}\times[t,t+1])}
   +\|w\|_{C^{\theta,\frac{\theta}{2}}
   (\ol{\Omega}\times[t,t+1])}\leq M
   \quad \mbox{for all}\ t\geq 1.
\end{align*}
\end{lem}
\begin{proof}
We fix 
$p\in I_1\cap I_2\cap (0,n)$, where $I_1$ and $I_2$ 
are the intervals defined in 
Lemmas \ref{lem;Lp;u} and \ref{lem;Lp;v}. 
Then thanks to Lemmas \ref{lem;Lp;u} and \ref{lem;Lp;v}, 
we can find $C_1>0$ such that 
\begin{align}\label{Lp;uandv}
  \lp{p}{u\cd}+\lp{p}{v\cd} \le C_1
 \quad \mbox{for all}\ t\in (0,\tmax). 
\end{align}
We first verify the $L^{\frac{np}{n-p}}$-estimate 
for $\nabla v$. 
Here for all $q\in (1,\infty)$, 
the standard elliptic regularity argument 
(see e.g., \cite[Theorem 19.1]{Friedman_1969}) 
leads to the existence of a constant $C_E(q)>0$ satisfying 
\begin{align}\label{regularity;w}
\wmp{2,q}{w\cd}\le C_E(q) (\lp{q}{u\cd}+\lp{q}{v\cd})
\quad\mbox{for all}\ t\in (0,\tmax). 
\end{align} 
Therefore a combination of 
\eqref{regularity;w} with \eqref{Lp;uandv} yields 
from the Sobolev embedding theorem that 
there exists $C_2>0$ such that 
\begin{align*}
 \lp{\frac{np}{n-p}}{\nabla v\cd}
 \le C_2
 \quad \mbox{for all}\ t\in (0,\tmax) 
\end{align*}
since $p<n$. 
We next establish the $L^\infty$-estimate for $u$. 
Since $p>\frac n2$, we can take $r\in (n,q)$ such that 
\begin{align*}
  p > \frac{nr}{n+r}. 
\end{align*}
We take $\vartheta >1$ satisfying 
\begin{align*}
  \frac{1}{\vartheta} 
  < \min\left\{
  1-\frac{r(n-p)}{np}, 
  \frac{q-r}{q}
  \right\}.
\end{align*}
Then $\vartheta':= \frac{\vartheta}{\vartheta-1}$ 
satisfies 
\begin{align*}
r\vartheta' < \frac{np}{n-p}. 
\end{align*}
Now for all $T\in (0,\tmax)$ 
we note that 
\begin{align*}
 A(T):= \sup_{t\in (0,t)}\lp{\infty}{u\cd}
\end{align*}
is finite. 
To obtain the estimate for $A(T)$ 
we put $t_0:=(t-1)_+$ and 
represent $u$ according to 
\begin{align}\notag
  u\cd 
  &= e^{(t-t_0)d_1\Delta}u(t_0) 
  -\chi_1\int_{t_0}^t e^{(t-s)d_1\Delta}\nabla \cdot 
  (u(\cdot,s)\nabla v(\cdot,s))\,ds 
\\\notag
  &\quad\, +\mu_1\int_{t_0}^t e^{(t-s)d_1\Delta}
  u(\cdot,s)(1-u(\cdot,s)-a_1v(\cdot,s))\,ds
\\\label{rep;u} 
  &=: u_1\cd + u_2\cd + u_3\cd 
\end{align}
for $t\in (0,\tmax)$. 
In the case that $t\le 1$, i.e., $t_0=0$, 
from the order preserving property of the 
Neumann heat semigroup we see that
\begin{align}\label{u1;t<1}
\lp{\infty}{u_1\cd}\le \lp{\infty}{u_0} 
\quad \mbox{for all}\ t\in 
(0,1]\cap (0,T). 
\end{align}
In the case that $t>1$ using the 
$L^p$-$L^q$ estimate for 
$(e^{\tau \Delta})_{\tau \ge 0}$ 
(see \cite[Lemma 1.3]{win_aggregationvs}) 
yields that there is $C_3>0$ such that 
\begin{align}\label{u1;t>1}
  \lp{\infty}{u_1\cd} 
  \le C_3\lp{p}{u(\cdot,t_0)}\le C_1C_3
  \quad \mbox{for all}\ t\in (1,T). 
\end{align}
Next due to a known smoothing property of 
$(e^{\tau \Delta})_{\tau \ge 0}$ (see \cite[Lemma 3.3]{FIWY_2016}), 
we can find $C_4>0$ such that 
\begin{align*} 
  \lp{\infty}{u_2\cd} 
  \le C_4  
  \sup_{t\in (0,T)}\lp{r}{u\cd\nabla v\cd}
  \int_{0}^1 \sigma^{-\frac 12-\frac n{2r}}\,d\sigma. 
\end{align*}
Noting from $r\vartheta'<\frac{np}{n-p}$ 
and \eqref{Lp;uandv} that 
\begin{align*}
  \lp{r}{u\cd\nabla v\cd}
  &\le \lp{r\vartheta}{u\cd}\lp{r\vartheta'}{\nabla v\cd}
\\
  &\le C_5\lp{\infty}{u\cd}^{1-\frac{p}{r\vartheta}}
  \lp{p}{u\cd}^{\frac{p}{r\vartheta}}\lp{\frac{np}{n-p}}{\nabla v\cd}
\\
  &\le C_1^{\frac{p}{r\vartheta}}C_2C_5 
       A(T)^{1-\frac{p}{r\vartheta}}
  \quad\mbox{for all}\ t\in (0,T) 
\end{align*}
with some $C_5>0$, 
we establish that there exists $C_6>0$ such that 
\begin{align}\label{u2}
  \lp{\infty}{u_2\cd}\le C_6 
  \quad \mbox{for all}\ t\in (0,T). 
\end{align}
Finally, the maximum principle together with the 
elementary inequality 
\[
  \mu_1 u(1-u-a_1v) 
  \le -\mu_1\left(u-\frac{1+\mu_1}{2\mu_1}\right)^2
  +\frac{(1+\mu_1)^2}{4\mu_1}
  \le \frac{(1+\mu_1)^2}{4\mu_1}
\] 
implies that there exists $C_7>0$ such that 
\begin{align}\label{u3}
u_3\cd \le C_7 
\quad \mbox{for all}\ t\in (0,T).
\end{align} 
Therefore a combination of \eqref{rep;u}, 
the nonnegativity of $u$ with 
\eqref{u1;t<1}, \eqref{u1;t>1}, \eqref{u2}, \eqref{u3} 
tells us that there exist $C_8,C_9>0$ such that 
\begin{align*}
  A(T) \le C_8 + C_9A(T)^{1-\frac{p}{r\vartheta}},
\end{align*}
which implies from $p<r\vartheta$ that 
\begin{align*}
A(T) \le C_{10} 
\quad \mbox{for all}\ T\in (0,\tmax)
\end{align*}
with some $C_{10}>0$. 
Thus we obtain the $L^\infty$-estimate for $u$. 
Similarly, we can verify the $L^\infty$-estimate for $v$. 
Then invoking \eqref{regularity;w}, 
we see that there exists $C_{11}>0$ such that 
\begin{align*}
\wmp{1,q}{w\cd}\le C_{11}
\quad \mbox{for all}\ t\in (0,\tmax), 
\end{align*}
which implies \eqref{Linfty;uvandw}. 
Moreover, known regularity arguments (see \cite[Proposition 2.3]{Cao-Wang-Yu_2016}) 
enable us to find $C_{12}>0$ and $\theta\in (0,1)$ 
satisfying 
\begin{align*}
   \|u\|_{C^{\theta,\frac{\theta}{2}}
   (\ol{\Omega}\times[t,t+1])}
   +\|v\|_{C^{\theta,\frac{\theta}{2}}
   (\ol{\Omega}\times[t,t+1])}
   +\|w\|_{C^{\theta,\frac{\theta}{2}}
   (\ol{\Omega}\times[t,t+1])}\leq C_{12}
   \quad \mbox{for all}\ t\geq 1, 
\end{align*}
which implies the end of the proof. 
\end{proof}
\begin{proof}[{\rm \bf Proof of Theorem \ref{mainth}}]
Lemma \ref{lem;Linfty;uvw} directly shows Theorem \ref{mainth}. 
\end{proof}
%

%
%
\section{Stabilization}
In this section we will establish stabilization of 
solutions to \eqref{cp}. 
Here we assume that 
there exists a unique global 
classical solution $(u,v,w)$ of \eqref{cp} 
satisfying 
\begin{align*}
 \|u\|_{C^{\theta,\frac{\theta}{2}}
 (\ol{\Omega}\times[t,t+1])} 
 + \|v\|_{C^{\theta,\frac{\theta}{2}}
 (\ol{\Omega}\times[t,t+1])} 
 + \|w\|_{C^{\theta,\frac{\theta}{2}}
 (\ol{\Omega}\times[t,t+1])}
 \le M
\quad \mbox{for all}\ t\ge 1 
\end{align*}
with some $M>0$. 
we first recall a important lemma 
for the proof of Theorems \ref{mainth2} 
and \ref{mainth3} (see \cite[Lemma 4.6]{HKMY_1}). 

\begin{lem}\label{lem;forconvergence}
Let $n\in C^{0}(\overline{\Omega}\times[0,\infty))$ satisfy that 
there exist constants $C^{*} > 0$ and $\theta^{*} > 0$ such that  
\begin{align*}
\|n\|_{C^{\theta^{*}, \frac{\theta^{*}}{2}}(\overline{\Omega}\times[t, t+1])} \leq C^{*} 
\quad \mbox{for all}\ t \geq 1.
\end{align*}
Assume that 
\begin{align*}
\int_0^\infty 
\int_\Omega (n(x,t)-N^\ast)^2\,dxdt < \infty 
\end{align*}
with some constant $N^\ast>0$. 
Then 
\begin{align*}
  n(\cdot,t)\to N^\ast 
  \quad 
  \mbox{in}\
  C^{0}(\overline{\Omega})
  \quad 
  \mbox{as}\ t\to\infty.
\end{align*}
\end{lem}

We next provide the following lemma which will 
be used to confirm that 
the assumption of Lemma \ref{lem;forconvergence} 
is satisfied. 

\begin{lem}\label{lem;justcal}
Let $a,b,c,d,e,f\in\Rone$. 
Suppose that 
\begin{align*}
  a>0,
\quad 
  d-\frac{b^2}{4a}>0,
\quad 
  f-\frac{c^2}{4a}-\frac{(2ae-bc)^2}{4a(4ad-b^2)}>0.
\end{align*}
Then 
\begin{align*}
ax^2+bxy+cxz+dy^2+eyz+fz^2\geq 0
\end{align*}
holds for all $x,y,z \in\Rone$. 
\end{lem}
\begin{proof}
Straightforward calculations lead to 
the conclusion 
of this lemma 
(for more details, see \cite[Lemma 3.2]{Mizukami_DCDSB}). 
\end{proof}

Finally, we give the following lemma which enables us to 
upgrade the $L^2$-convergence rate 
to $L^\infty$-convergence rate. 
\begin{lem}\label{L2toLinfty}
 Let $(\ubar,\vbar,\wbar)\in \mathbb{R}^3$ be 
 a solution to \eqref{cp}. 
 Assume that there exists a decreasing function 
 $h:[0,\infty)\to\mathbb{R}$ satisfying  
 \begin{align*}
   \lp{2}{u\cd-\ubar}
   +\lp{2}{v\cd-\vbar}
   \le h(t)
   \quad \mbox{for all}\ t>0.
 \end{align*} 
 Then there exists $C>0$ such that 
 \begin{align*}
   \lp{\infty}{u\cd-\ubar}
   + \lp{\infty}{v\cd-\vbar}
   + \lp{\infty}{w\cd-\wbar}
   \le Ch(t-1)^{\frac{1}{n+1}}
 \end{align*}
for all $t>1$. 
\end{lem}
\begin{proof}
For all $p>2$ 
we first obtain from the H\"older inequality 
that 
\begin{align*}
  \lp{p}{f}\le 
  \lp{\infty}{f}^{1-\frac{2}{p}}\lp{2}{f}^{\frac{2}{p}}
\end{align*}
holds for all $f\in L^{\infty}(\Omega)$, which means 
from the boundedness of $u,v$ that 
\begin{align*} 
  \lp{p}{u\cd - \ubar}
  + \lp{p}{v\cd - \vbar}
  \le C_1(p) h(t)^{\frac 2p}
\quad \mbox{for all}\ t>0
\end{align*}
with some $C_1(p)>0$. 
Here \eqref{regularity;w} enables us to see that 
\begin{align*}
 \| w\cd -\wbar\|_{W^{2,2n+2}(\Omega)} 
 \le C_E(2n+2)C_1(2n+2) h(t)^{\frac{1}{n+1}} 
 \quad \mbox{for all}\ t>0. 
\end{align*}
Thus we have that there is $C_2>0$ such that 
\begin{align*}
  \lp{2n+2}{\nabla w\cd} 
  \le C_2 h(t)^{\frac{1}{n+1}}
  \quad \mbox{for all}\ t>0. 
\end{align*}
Then by using a similar argument 
as in the proof of \cite[Lemma 3.6]{Bai-Winkler_2016} 
we infer that there exists $C_3>0$ such that 
\begin{align*}
  \lp{\infty}{u\cd-\ubar}
  + \lp{\infty}{v\cd-\vbar}
  \le C_3h(t-1)^{\frac{1}{n+1}}
  \quad \mbox{for all}\ t>1. 
\end{align*}
Finally, since $(\ubar,\vbar,\wbar)$ satisfies 
\[
\alpha \ubar+\beta \vbar -\gamma \wbar=0, 
\]
we can apply the maximum principle to 
\[
  -\Delta (w-\wbar) + \gamma (w - \wbar) 
  = \alpha (u-\ubar) + \beta (v-\vbar), 
\]
and hence obtain  
the existence of a constant $C_4>0$ such that 
\begin{align*}
 \lp{\infty}{w\cd-\wbar} 
 &\le C_4 
 (\lp{\infty}{u\cd -\ubar}+\lp{\infty}{v\cd-\vbar})
\\ 
 &\le C_3C_4 h(t-1)^{\frac 1{n+1}} 
 \quad \mbox{for all}\ t>1, 
\end{align*}
which concludes the proof of this lemma. 
\end{proof}

%
%

\subsection{Convergence. Case 1: $a_1,a_2\in (0,1)$}
In this subsection we establish stabilization 
in the case that $a_1,a_2\in (0,1)$. 
We first confirm that the assumption of Lemma 
\ref{lem;forconvergence} are satisfied. 

%
%
%
%
%
\begin{lem}\label{lem;energy}
Assume that 
\eqref{case1condition1}{\rm --}\eqref{case1condition3} 
are satisfied. 
Then there exist a nonnegative function 
$E_1:(0,\infty)\to \mathbb{R}$ and 
a constant $\ep>0$ such that 
\begin{align}\label{purpose;case1}
  \frac d{dt}E_1(t) 
  \le - \ep \into 
  \left[
  (u\cd -\uast)^2
  +(v\cd-\vast)^2
  + (w\cd-\wast)^2
  \right]
\end{align}
holds for all $t>0$. 
Moreover, there exists $C>0$ satisfying 
\begin{align*}
 \int_0^\infty \into (u-\uast)^2 
 + \int_0^\infty \into (v-\vast)^2
 + \int_0^\infty \into (w-\wast)^2
 \le C. 
\end{align*}
\end{lem}
\begin{proof}
Let $\delta_1>0$ be a constant defined in 
\eqref{case1condition1}{\rm --}\eqref{case1condition3}. 
First we shall show that the function $E_1:(0,\infty)\to \mathbb{R}$ defined as 
\begin{align}\label{defE}
 E_1:= \into \left(
 u-\uast-\uast \log \frac{u}{\uast}
 \right)
 + \frac{a_1\mu_1\delta_1}{a_2\mu_2}
 \into \left(
 v-\vast-\vast \log \frac{v}{\vast}
 \right)
\end{align}
satisfies that \eqref{purpose;case1} holds 
for all $t>0$ with some $\ep>0$. 
From straightforward calculations 
we infer that 
\begin{align}\notag
  \frac d{dt} E_1(t) 
  &= -\mu_1 \into (u\cd-\uast)^2 
  - (1+\delta_1)a_1\mu_1\into (u\cd-\uast)(v\cd-\vast) 
\\\notag &\quad\,
  - \frac{a_1\mu_1\delta_1}{\mu_2}\into (v\cd-\vast)^2 
  - d_1\uast\into \frac{|\nabla u\cd|^2}{u^2} 
\\\notag &\quad\, 
  + \uast \chi_1\into \frac{\nabla u\cd\cdot \nabla w\cd}{u} 
  - \frac{d_2a_1\mu_1 \vast\delta_1}{a_2\mu_2}
    \into \frac{|\nabla v\cd|^2}{v^2}
\\\label{derE} &\quad \,
  + \frac{a_1\mu_1\vast \chi_2\delta_1}{a_2\mu_2}
    \into \frac{\nabla v\cd\cdot \nabla w\cd}{v}
  \quad \mbox{for all}\ t>0. 
\end{align}
Here in light of \eqref{case1condition1}{\rm --}\eqref{case1condition3} 
we can take $\delta_2>0$ satisfying 
\begin{align*}
  \frac{\uast \chi_1^2}{4d_1}
  < \delta_2 
  < \frac{d_3a_1\mu_1\gamma (4\delta_1 - (1+\delta_1)^2a_1a_2)}
    {(1+\delta_1)(a_1\alpha^2 \delta_1 + a_2\beta^2
    -(1+\delta_1) a_1a_2\alpha \beta)}
\end{align*}
and 
\begin{align*}
  \frac{a_1\mu_1\vast \chi_2^2}{4d_2a_2\mu_2}
  < \delta_2 
  < \frac{d_3a_1\mu_1\gamma (4\delta_1 - (1+\delta_1)^2a_1a_2)}
    {(1+\delta_1)(a_1\alpha^2 \delta_1 + a_2\beta^2
    -(1+\delta_1) a_1a_2\alpha \beta)}. 
\end{align*}
Invoking the Young inequality, we obtain that 
\begin{align}\label{Young;1}
  \uast\chi_1\into \frac{\nabla u\cdot \nabla w}{u}
  \le \frac{{\uast}^2\chi_1^2}{4\delta_2}
      \into \frac{|\nabla u|^2}{u^2} 
    + \delta_2 \into |\nabla w|^2
\end{align}
and 
\begin{align}\label{Young;2}
  \frac{a_1\mu_1 \vast\chi_2\delta_1}{a_2\mu_2} 
  \into \frac{\nabla v\cdot \nabla w}{v} 
  \le \frac{a_1^2\mu_1^2 {\vast}^2\chi_2^2\delta_1}{4\delta_2}
  \into \frac{|\nabla v|^2}{v^2}
  + \delta_1\delta_2 \into |\nabla w|^2. 
\end{align}
Therefore since the definition of $\delta_2$ yields 
\begin{align*}
     d_1 - \frac{\uast \chi_1^2}{4\delta_2}>0 
     \quad \mbox{and}\quad 
      d_2-\frac{a_1\mu_1\vast\chi_2^2}{4a_2\mu_2\delta_2}>0, 
\end{align*} 
a combination of \eqref{derE} with 
\eqref{Young;1} 
and \eqref{Young;2} implies 
\begin{align*}
  \frac{d}{dt}E_1(t) 
  &\le  
   -\mu_1 \into (u\cd-\uast)^2 
  - (1+\delta_1)a_1\mu_1\into (u\cd-\uast)(v\cd-\vast) 
\\ 
&\quad\,
  - \frac{a_1\mu_1\delta_1}{\mu_2}\into (v\cd-\vast)^2
   + (1+\delta_1)\delta_2 \into |\nabla w\cd|^2
  \quad \mbox{for all}\ t>0. 
\end{align*}
Noting from the third equation in \eqref{cp} that 
\begin{align*}
  \into |\nabla w|^2 
  = \frac \alpha{d_3} \into (u-\uast)(w-\wast)  
    + \frac \beta{d_3} \into (v-\vast)(w-\wast)
    - \frac \gamma{d_3} \into (w-\wast)^2, 
\end{align*}
we establish that 
\begin{align*}
    \frac{d}{dt}E_1(t) \le F_1(t) 
    \quad \mbox{for all}\ t>0, 
\end{align*}
where  
\begin{align*}
  F_1(t)&:=
  -\mu_1 \into (u\cd-\uast)^2 
  - (1+\delta_1)a_1\mu_1\into (u\cd-\uast)(v\cd-\vast) 
\\ 
&\quad\,
  - \frac{a_1\mu_1\delta_1}{\mu_2}\into (v\cd-\vast)^2
  + \frac{\alpha(1+\delta_1)\delta_2}{d_3} 
    \into (u\cd-\uast)(w\cd-\wast)
\\ 
&\quad \,
  + \frac{\beta(1+\delta_1)\delta_2}{d_3} 
    \into (v\cd-\vast)(w\cd-\wast)
  - \frac{\gamma(1+\delta_1)\delta_2}{d_3} 
    \into (w\cd-\wast)^2
\end{align*}
for all $t>0$. 
In order to see \eqref{purpose;case1} 
we will show that 
\begin{align*}
  F_1(t) \le 
  - \ep \into 
  \left[
  (u\cd -\uast)^2
  +(v\cd-\vast)^2
  + (w\cd-\wast)^2
  \right]
\end{align*}
with some $\ep>0$ by using Lemma \ref{lem;justcal}. 
To confirm that the assumption of 
Lemma \ref{lem;justcal} is satisfied 
we put 
\begin{align*}
 g_1(\ep)&:= \mu_1 -\ep, \qquad 
 g_2(\ep) := \frac{a_1\mu_1\delta_1}{a_2} - \ep 
             -\frac{(1+\delta_1)^2a_1^2\mu_1^2}{4(\mu_1-\ep)}, 
\\
 g_3(\ep)&:= \frac{\gamma (1+\delta_1)\delta_2}{d_3} -\ep 
             - \frac{\alpha^2(1+\delta_1)^2\delta_2^2}{4d_3^2(\mu_1-\ep)}
\\ &\quad \
             - \frac{\left(2(\mu_1-\ep)\beta(1+\delta_1)\delta_2
                     - (1+\delta_1)^2a_1\mu_1\alpha\delta_1
                     \right)^2
               }{4d_3^2(\mu_1-\ep)
               \left(
               4(\mu_1-\ep)(\frac{a_1\mu_1\delta_1}{a_2}-\ep)
               -(1+\delta_1)^2a_1\mu_1\right)}
\end{align*}
for $\ep>0$, and shall see that there exists $\ep_1>0$ such that 
$g_i(\ep_1)>0$ for $i=1,2,3$. 
Here $g_1(0)=\mu_1>0$ obviously holds, 
and the condition \eqref{case1condition1} implies that 
\begin{align*}
 g_2(0)=\frac{a_1\mu_1(4\delta_1 
 - (1+\delta_1)^2a_1a_2)}{4a_2}>0. 
\end{align*}
Moreover, aided by the definition of $\delta_2$, 
we can obtain that 
\begin{align*}
  g_3(0)&= \frac{\gamma (1+\delta_1)\delta_2}{d_3} 
             - \frac{\alpha^2(1+\delta_1)^2\delta_2^2}{4d_3^2\mu_1}
             - \frac{\left(2\mu_1\beta(1+\delta_1)\delta_2
                     - (1+\delta_1)^2a_1\mu_1\alpha\delta_1
                     \right)^2
               }{4d_3^2\mu_1
               \left(
               4\mu_1\frac{a_1\mu_1\delta_1}{a_2}
               -(1+\delta_1)^2a_1\mu_1\right)}
\\
       &= 
       (1+\delta_1)\delta_2\left(
       \frac{\gamma}{d_3} 
        - \frac{
        (1+\delta_1)
        (\alpha^2a_1\delta_1 + 
        a_2\beta^2 - (1+\delta_1)a_1a_2\alpha\beta)
         }{d_3^2a_1\mu_1(4\delta_1-(1+\delta_1)^2a_1a_2)}
         \delta_2
         \right)>0.
\end{align*}
Therefore a combination of the above inequalities 
and the continuity argument implies 
that there exists $\ep_1>0$ such that $g_i(\ep_1)>0$ 
hold for $i=1,2,3$. 
Thus Lemma \ref{lem;justcal} derives that 
\begin{align*}
  F_1(t) \le -\ep_1 \into 
  \left[
  (u\cd-\uast)^2+(v\cd-\vast)^2+(w\cd-\wast)^2
  \right]
  \quad \mbox{for all}\ t>0, 
\end{align*}
which yields that \eqref{purpose;case1} holds 
for all $t>0$. Then since from the Taylor formula 
$E_1$ is a nonnegative function for $t>0$ 
(more details, see \cite[Lemma 3.2]{Bai-Winkler_2016}), 
integrating \eqref{purpose;case1} over $(0,\infty)$ 
concludes the proof of this lemma. 
\end{proof} 
%
%
%
%
\begin{lem}\label{lem;con;sol}
Assume that 
\eqref{case1condition1}{\rm --}\eqref{case1condition3} 
are satisfied. Then 
\begin{align*}
 \lp{\infty}{u\cd-\uast}
 +\lp{\infty}{v\cd-\vast}
 +\lp{\infty}{w\cd-\wast}
 \to 0
\quad 
 \mbox{as}\ t\to \infty. 
\end{align*}
\end{lem}
\begin{proof}
A combination of 
Lemmas \ref{lem;forconvergence} and \ref{lem;energy} 
implies this lemma. 
\end{proof}
%
%
%
%
Next we desire to establish convergence rates 
for the solution of \eqref{cp}. We note that 
in view of Lemma \ref{L2toLinfty} it is sufficient to 
confirm the $L^2$-convergence rates for the solution. 
\begin{lem}\label{lem;decay;L2}
Assume that 
\eqref{case1condition1}{\rm --}\eqref{case1condition3} 
are satisfied. 
Then there exist $C > 0$ and $\ell > 0$ such that 
\begin{align*}
\lp{2}{u\cd-\uast}
+\lp{2}{v\cd-\vast}
\le Ce^{-\ell t}
\quad \mbox{for all}\ t>0.
\end{align*}
\end{lem}
\begin{proof}
Aided by Lemma \ref{lem;con;sol} and 
the L'H\^opital theorem, a similar argument 
as in the proof of \cite[Lemma 3.7]{Bai-Winkler_2016} 
(or \cite[Proof of Theorem 1.2]{Mizukami_DCDSB}) 
derives that there exist $C_1,C_2>0$ and $t_0>0$ 
such that for all $t>t_0$, 
\begin{align}\label{goodrela;case1}
 C_1\left(
 \into (u-\uast)^2 
 + \into (v-\vast)^2
 \right) 
 \le 
 E_1
 \le 
 C_2 \left(
 \into (u-\uast)^2
 + \into (v-\vast)^2
 \right), 
\end{align}
where $E_1$ is the function defined as \eqref{defE}. 
Therefore we obtain from \eqref{purpose;case1} that 
\begin{align*}
\frac d{dt}E_1(t) \le - C_3 E_1(t) 
\quad \mbox{for all}\ t>t_0 
\end{align*}
with some $C_3>0$, which implies that 
there exists $C_4>0$ such that 
\begin{align}\label{Eineq;case1;ex}
  E_1(t) \le C_4 e^{-C_3t}
  \quad \mbox{for all}\ t>0. 
\end{align}
Thus a combination of 
\eqref{goodrela;case1} and \eqref{Eineq;case1;ex} 
yields that 
\begin{align*}
  \into (u-\uast)^2 
  + \into (v-\vast)^2
  \le \frac{C_4}{C_1}e^{-C_3 t},
\end{align*}
which concludes the proof of this lemma. 
\end{proof}
%
%
%
%
\subsection{Convergence. Case 2: $a_1\ge 1 > a_2$}

In this subsection we will obtain stabilization 
in the case that $a_1\ge 1 > a_2$. 
In this case we also have to confirm that the assumption 
of Lemma \ref{lem;forconvergence} is satisfied. 

\begin{lem}\label{lem;energy;case2}
Assume that 
\eqref{case2condition1}{\rm --}\eqref{case2condition2} 
are satisfied. 
Then there exist a nonnegative function 
$E_2:(0,\infty)\to \mathbb{R}$ and 
constants $\ep_1\ge 0$ and $\ep_2> 0$ such that 
\begin{align}\label{purpose;case2}
  \frac d{dt}E_2(t) 
  \le 
  -\ep_1\into u\cd
  - \ep_2 \into 
  \left[
  u\cd^2
  +(v\cd-1)^2
  + \left(w\cd-\frac{\beta}{\gamma}\right)^2
  \right]
\end{align}
holds for all $t>0$. 
Moreover, there exists $C>0$ satisfying 
\begin{align*}
 \int_0^\infty \into u^2 
 + \int_0^\infty \into (v-1)^2
 + \int_0^\infty \into \left(w-\frac{\beta}{\gamma}\right)^2
 \le C. 
\end{align*}
\end{lem}
\begin{proof}
Let $\delta_1>0$ and $a_1'\in [1,a_1]$ 
be a constant defined in 
\eqref{case2condition1}{\rm --}\eqref{case2condition2}. 
We first show that the function 
$E_2:(0,\infty)\to \mathbb{R}$ defined as 
\begin{align}\label{defE2}
 E_2:= \into u
 + \frac{a_1'\mu_1\delta_1}{a_2\mu_2}
 \into \left(
 v-1-\log {v}
 \right)
\end{align}
fulfils \eqref{purpose;case2} 
for all $t>0$ with $\ep_1:=a_1'-1$ and some $\ep_2>0$. 
Noting from the relation $a_1'\le a_1$ that 
\begin{align*}
  u(1-u-a_1v)&\le u(1-u-a_1'v)
  \\
  &=  - \mu_1u^2-a_1'\mu_1 u(v-1) -(a_1'-1)\mu_1u,
\end{align*}
from straightforward calculations 
we derive that 
\begin{align}\notag
  \frac d{dt} E_2(t) 
  &= -\mu_1 \into u\cd^2 
  - (1+\delta_1)a_1'\mu_1\into u\cd(v\cd-1) 
\\\notag &\quad\,
  - \frac{a_1'\mu_1\delta_1}{\mu_2}\into (v\cd-1)^2 
  -(a_1'-1)\into u\cd
\\\label{derE2} &\quad \,
  - \frac{d_2a_1'\mu_1 \delta_1}{a_2\mu_2}
    \into \frac{|\nabla v\cd|^2}{v^2}
  + \frac{a_1'\mu_1\chi_2\delta_1}{a_2\mu_2}
    \into \frac{\nabla v\cd\cdot \nabla w\cd}{v}
\end{align}
for all $t>0$. 
Here thanks to 
\eqref{case2condition1}{\rm --}\eqref{case2condition2}, 
we can take $\delta_2>0$ such that  
\begin{align*}
  \frac{a_1'\mu_1 \chi_2^2\delta_1}{4d_2a_2\mu_2}
  < \delta_2 
  < \frac{d_3a_1'\mu_1\gamma (4\delta_1 - (1+\delta_1)^2a_1'a_2)}
    {a_1'\alpha^2 \delta_1 + a_2\beta^2
    -(1+\delta_1) a_1'a_2\alpha \beta}. 
\end{align*}
Invoking the Young inequality, we obtain that 
\begin{align*}
  \frac{a_1'\mu_1 \chi_2\delta_1}{a_2\mu_2} 
  \into \frac{\nabla v\cdot \nabla w}{v} 
  \le \frac{a_1'^2\mu_1^2 \chi_2^2\delta_1^2}{4\delta_2}
  \into \frac{|\nabla v|^2}{v^2}
  + \delta_2 \into |\nabla w|^2. 
\end{align*}
Therefore since the definition of $\delta_2$ yields 
\begin{align*}
      d_2-\frac{a_1'\mu_1\chi_2^2\delta_1}
      {4a_2\mu_2\delta_2}>0, 
\end{align*}
the equation \eqref{derE2} implies that
\begin{align*} 
  \frac{d}{dt}E_2(t) 
  &\le  
   -\ep_1 \into u
   -\mu_1 \into u\cd^2 
  - (1+\delta_1)a_1'\mu_1\into 
  u\cd(v\cd-1) 
\\ 
&\quad\,
  - \frac{a_1'\mu_1\delta_1}{\mu_2}\into (v\cd-1)^2
   + \delta_2 \into |\nabla w\cd|^2
\end{align*}
for all $t>0$, where $\ep_1= a_1'-1$. 
Noting from the third equation in \eqref{cp} that 
\begin{align*}
  \into |\nabla w|^2 
  = \frac \alpha{d_3} 
    \into u\left(w-\frac{\beta}{\gamma}\right)  
    + \frac \beta{d_3} 
    \into (v-1)\left(w-\frac{\beta}{\gamma}\right)
    - \frac \gamma{d_3} 
    \into \left(w-\frac{\beta}{\gamma}\right)^2, 
\end{align*}
we establish that for all $t>0$, 
\begin{align*}
    \frac{d}{dt}E_2(t) \le -\ep_1 \into u + F_2(t) 
    \quad \mbox{for all}\ t>0, 
\end{align*}
where  
\begin{align*}
  F_2(t)&:=
  -\mu_1 \into u\cd^2 
  - (1+\delta_1)a_1\mu_1\into u\cd(v\cd-1) 
\\ 
&\quad\,
  - \frac{a_1\mu_1\delta_1}{\mu_2}\into (v\cd-1)^2
  + \frac{\alpha(1+\delta_1)\delta_2}{d_3} 
    \into u\cd\left(w\cd-\frac{\beta}{\gamma}\right)
\\ 
&\quad \,
  + \frac{\beta(1+\delta_1)\delta_2}{d_3} 
    \into (v\cd-1)\left(w\cd-\frac{\beta}{\gamma}\right)
  - \frac{\gamma(1+\delta_1)\delta_2}{d_3} 
    \into \left(w\cd-\frac{\beta}{\gamma}\right)^2.
\end{align*}
Then by using the same argument as in the proof of 
Lemma \ref{lem;energy} we can see that 
\begin{align*}
  F_2(t) \le 
  - \ep_2 \into 
  \left[
  u\cd^2
  +(v\cd-1)^2
  + \left(w\cd-\frac{\beta}{\gamma}\right)^2
  \right]
\quad\mbox{for all}\ t>0
\end{align*}
with some $\ep_2>0$, 
which means that \eqref{purpose;case2} holds 
with $\ep_1=a_1'-1\ge 0$ and $\ep_2 > 0$. 
\end{proof}
%
%
%
%
Then we will establish the convergence result for 
the solution to \eqref{cp} 
in the case that $a_1\ge 1>a_2$. 
\begin{lem}\label{lem;con;sol;case2}
Assume that 
\eqref{case2condition1}{\rm --}\eqref{case2condition2} 
are satisfied. 
Then we have
\begin{align*}
 \lp{\infty}{u\cd}
 +\lp{\infty}{v\cd-1}
 +\left\|w\cd-\frac{\beta}{\gamma}\right\|_{L^\infty(\Omega)}
 \to 0
\quad
 \mbox{as}\ t\to \infty. 
\end{align*}
\end{lem}
\begin{proof}
A combination of 
Lemmas \ref{lem;forconvergence} and \ref{lem;energy;case2} 
implies this lemma. 
\end{proof}
%
%
%
%
Finally, we shall show two lemmas which 
give asymptotic behavior 
in the case that $a_1> 1>a_2$. 
\begin{lem}\label{lem;decay;L2;case2}
Let $a_1>1$ and $a_2\in (0,1)$. 
Assume that 
\eqref{case2condition1}{\rm --}\eqref{case2condition2} 
are satisfied with $\delta_1>0$ and $a_1'>0$. 
Then there exist $C>0$ and $\ell > 0$ satisfying 
  \begin{align*}
  \|u(t)\|_{L^{2}(\Omega)}
  + \|v(t)-1\|_{L^{2}(\Omega)}
  \leq Ce^{-\ell t}
  \quad\mbox{for all}\ t>0. 
  \end{align*}
\end{lem}
\begin{proof}
In the case that $a_1\ge 1$ and $a_2\in (0,1)$ 
a similar argument as in the proof of 
\cite[Lemmas 4.3]{Mizukami_DCDSB} 
enables us to see that there exist 
$C_1,C_2>0$ and $t_0>0$ such that 
\begin{align*}
C_1 h_1(t) \le E_2(t) \le C_2 h_1(t) 
\quad 
\mbox{for all}\ t>t_0,
\end{align*}
where $E_2$ is the function defined as \eqref{defE2} 
and 
\begin{align*}
  h_1(t):= \into u\cd^2 + \into (v\cd-1)^2 
         + (a_1'-1)\into u\cd. 
\end{align*}
Thus a combination of the above inequality and 
\eqref{purpose;case2} means that  
\begin{align*}
 \frac{d}{dt}E_2(t) \le -C_3E_2(t) 
\end{align*}
holds for all $t>t_0$, which 
together with the same argument as 
in the proof of Lemma \ref{lem;decay;L2} leads to 
the conclusion of this lemma 
in the case that $a_1>1$ and $a_2\in (0,1)$. 
\end{proof}
%
%
%
%
%
\begin{lem}\label{lem;decay;L2;case3}
Let $a_1=1$ and $a_2\in (0,1)$. 
Assume that 
\eqref{case2condition1}{\rm --}\eqref{case2condition2} 
are satisfied. 
Then there exist $C>0$ and $\ell>0$ satisfying 
  \begin{align*}
  \|u(t)\|_{L^{2}(\Omega)}
  + \|v(t)-1\|_{L^{2}(\Omega)}
  \leq \frac{C}{\sqrt{t+2}}
  \quad \mbox{for all}\ t>0. 
  \end{align*}
\end{lem}
\begin{proof}
First we can verify from the same argument as 
in the proof of \cite[Lemma 3.7]{Mizukami_DCDSB} 
that 
there exist $C_4,C_5>0$ and $t_1>0$ such that  
\begin{align}\label{v;ineq;energy}
  C_4 \into (v\cd-1)^2 
\le 
  \into (v\cd-1-\log v\cd)
\le
  C_5 \into (v\cd-1)^2 
\end{align}
for all $t>t_1$. 
Hence it follows from the Cauchy--Schwarz inequality 
and the boundedness of $v$ that 
\begin{align*}
  E_2(t) 
  &\le \into u\cd 
  + \frac{a_1'\mu_1\delta_1}{a_2\mu_2}
    \into (v\cd-1)^2
\\
  &\le C_6\left(\into u\cd^2\right)^{\frac{1}{2}}
  + C_6 \left(\into (v\cd-1)^2\right)^{\frac 12}
\\
  &\le \sqrt{2}C_6 
  \left(\into u\cd^2 + \into (v\cd-1)^2\right)^{\frac{1}{2}}
\quad \mbox{for all}\ t>t_1, 
\end{align*}
which implies from \eqref{purpose;case2} that 
\begin{align*}
  E_2(t) \le - C_7 E_2(t)^2
  \quad \mbox{for all}\ t>t_1. 
\end{align*}
Therefore we can find $C_8>0$ such that 
\begin{align*}
  E_2(t) \le \frac{C_8}{t+2} 
  \quad \mbox{for all}\ t>t_1. 
\end{align*}
Therefore thanks to the boundedness of $u$ and 
\eqref{v;ineq;energy}, 
we obtain that 
\begin{align*}
  \into u\cd^2 + \into (v\cd-1)^2
  \le C_9 E_2(t) \le \frac{C_8C_9}{t+2} 
  \quad \mbox{for all}\ t>t_1, 
\end{align*}
which proves this lemma. 
\end{proof}
\begin{proof}[{\rm \bf Proof of Theorems \ref{mainth2} and \ref{mainth3}}]
A combination of Lemmas \ref{lem;decay;L2}, \ref{lem;decay;L2;case2}, 
\ref{lem;decay;L2;case3} 
and \ref{L2toLinfty} immediately leads to 
the conclusions of these theorems. 
\end{proof}
%


\bibliographystyle{plain}

\end{document}